\documentclass[10pt,journal,compsoc]{IEEEtran}

\ifCLASSOPTIONcompsoc

  \usepackage[nocompress]{cite}
\else

  \usepackage{cite}
\fi

\usepackage{amsmath}
\usepackage{amssymb}
\usepackage{bbm}
\usepackage{amsthm}
\usepackage{graphicx}

\newtheorem{theorem}{Theorem}

\newtheorem{proposition}{Proposition}
\newtheorem{lemma}{Lemma}
\newtheorem{remark}{Remark}

\theoremstyle{definition}
\newtheorem{definition}{Definition}[section]

\newcommand{\dx}{\, \mathrm d x}
\newcommand{\dy}{\, \mathrm d y}

\newcommand{\vertiii}[1]{{\left\vert\kern-0.25ex\left\vert\kern-0.25ex\left\vert #1 
    \right\vert\kern-0.25ex\right\vert\kern-0.25ex\right\vert}}

\usepackage{color}

\newcommand{\until}[1]{\{1,\dots, #1\}}

\begin{document}
\title{The Laplacian Spectrum of Large Graphs Sampled from Graphons}
\author{Renato~Vizuete,~\IEEEmembership{Graduate Student Member,~IEEE,} 
        Federica~Garin,~\IEEEmembership{Member,~IEEE,}
        and~Paolo~Frasca,~\IEEEmembership{Senior Member,~IEEE}
\thanks{R. Vizuete is with the Universit\'{e} Paris-Saclay, CNRS, CentraleSup\'{e}lec, Laboratoire des signaux et syst\`{e}mes, 91190, Gif-sur-Yvette, France, and also with the Univ.\ Grenoble Alpes, CNRS, Inria, Grenoble INP, GIPSA-lab, F-38000 Grenoble, France. F.~Garin and P.~Frasca are with Univ.\ Grenoble Alpes, CNRS, Inria, Grenoble INP, GIPSA-lab, F-38000 Grenoble, France. Emails: 
{\tt\small renato.vizuete@l2s.centralesupelec.fr},
{\tt\small federica.garin@inria.fr}, \protect\\{\tt\small paolo.frasca@gipsa-lab.fr}
.}

\thanks{This work has been partially supported by the French National Science Foundation through grants LabEx PERSYVAL-Lab (ANR-11-LABX-0025-01) and HANDY (ANR-18-CE40-0010).}}

\IEEEtitleabstractindextext{%
\begin{abstract}
This paper studies the Laplacian spectrum and the average effective resistance of (large) graphs that are sampled from graphons. Broadly speaking, our main finding is that the Laplacian eigenvalues of a large dense graph can be effectively approximated by using the degree function of the corresponding graphon. More specifically, we show how to approximate the distribution of the Laplacian eigenvalues and the average effective resistance (Kirchhoff index) of the graph. For all cases, we provide explicit bounds on the approximation errors and derive the asymptotic rates at which the errors go to zero when the number of nodes goes to infinity.
Our main results are proved under the conditions that the graphon is piecewise Lipschitz and bounded away from zero. 
\end{abstract}

}
\maketitle

\IEEEdisplaynontitleabstractindextext
\IEEEpeerreviewmaketitle
\IEEEraisesectionheading{\section{Introduction}\label{sec:introduction}}

\IEEEPARstart{T}{he} study of large networks has been a focus of attention in recent years due to the increasing relevance of large networks in multiple fields, from computer science and engineering to biology, economics and sociology. 
Large networks require specific methods not only because of their size but also because their topologies are often known with large uncertainties and can dynamically evolve with time. A prominent tool to approach large networks is the concept of \textit{graphon}, developed in \cite{lovasz2006limits,lovasz2012large,janson2013graphons} more than a decade ago. Graphons are infinite-dimensional representations of ``families'' of graphs and limit objects of convergent graph sequences. Their handy mathematical properties are allowing for a rapidly increasing number of applications in multiple fields, including game theory \cite{parise2019graphon,gao2020lqg}, signal processing~\cite{ruiz2019graphon,ruiz2020graphon}, control theory \cite{gao2019optimal,gao2019graphon}, and the study of diffusion processes~\cite{petit2019random} and epidemics~\cite{gao2019spectral,vizuete2019graphon} on graphs. These applications are demonstrating that graphons can also be a versatile tool to study dynamics on large networks. 

Since the concept of graphon is inherently related to the adjacency matrix of graphs, its applications have essentially focused on cases when the adjacency matrix is the object of study. However, this can be limiting for some applications, because many network-based dynamics are instead better described by using the Laplacian matrix of the graph~\cite{bullo2019network}.
Other applications of the Laplacian matrix include spectral clustering \cite{von2008consistency}, combinatorial optimization \cite{mohar1993eigenvalues}, and signal processing \cite{shuman2013emerging}. 

The spectrum of the Laplacian matrix encodes relevant properties of the network, including its connectivity that can be measured by its spectral gap, that is, the magnitude of its smallest nonzero eigenvalue (if the spectral gap is small, the graph is poorly connected). The Laplacian spectrum also has an important role in the study of graphs by associating an electrical network to them \cite{dorfler2018electrical}.  
Representing graphs as networks of resistors is a classical tool in graph theory with a large range of applications such as the analysis of random walks \cite{doyle2000random,lovasz1993random}, consensus algorithms \cite{lovisari2013resistance,vizuete2021noise}, and distributed estimation algorithms \cite{barooah2007estimation}. 
In this context, a key graph property is the average effective resistance, which can be written as a function of the eigenvalues of the Laplacian matrix. The average effective resistance, also known as Kirchhoff index, can also be used to evaluate the connectivity of a network: small values imply well-connected networks. 
The average effective resistance can be computed, exactly or approximately, for many specific graph topologies, including toroidal graphs \cite{rossi2015average}, $d$-dimensional grids \cite{AC-PR-WR-RS-PT:96}, and other graphs with geometric properties \cite{lovisari2013resistance}. 

However, closed-form expressions for the eigenvalues are not available beyond few academic examples and therefore are of little help for graphs that represent real networks. Actually, real networks are often very large and their size can make the numerical computation of eigenvalues impractical, even accounting for the the recent developments of fast Laplacian solvers \cite{saad2011numerical,vishnoi2012laplacian,spielman2014nearly,kyng2016approximate}. 
Even worse, the topology of the network may not be fully known or be dynamically evolving, therefore preventing the direct application of numerical methods.

In this paper we demonstrate that, for large dense networks that are well described by graphons, properties of these limit objects can be used to provide useful approximations of the Laplacian properties.
Work in this direction has so far been limited to spectral clustering~\cite{luxburg2005limits,von2008consistency} and random walks on graphons~\cite{petit2019random}: we offer here a careful analysis of the approximation properties of graphons for the Laplacian eigenvalues.

In the case of large dense networks, the Laplacian matrix can be seen as a perturbation of the degree matrix of the graph \cite{zhan2010distributions,hata2017localization}, such that the contribution of the adjacency matrix to the Laplacian spectrum is small. Therefore, the distribution of the eigenvalues of the Laplacian matrix is close to the distribution of the degrees. At the same time, the degree function of a graphon is closely related to the degrees of the sampled graphs \cite{avella2018centrality}. 
Combining these two facts, it becomes natural to use the degree function of a graphon to approximate the Laplacian spectrum and, consequently, the average effective resistance of (large) graphs that are sampled from that graphon. 

Motivated by this informal discussion, the objective of this paper is to use characteristics of graphons for the analysis of the spectrum of the Laplacian matrix of graphs that are sampled from graphons. More precisely, our contribution is showing that the degree function of the graphon can be used to approximate the distribution of the Laplacian eigenvalues (Theorem~\ref{thm:spectrum}) and the average effective resistance (Theorem~\ref{main_theorem}). These results will be proved under the technical assumptions of the graphon being piecewise Lipschitz and bounded away from zero. 

The rest of this paper is structured as follows. Section~2 introduces the necessary preliminaries about graphons and sampled graphs. Section~3 presents our main results regarding both the Laplacian spectrum and the average effective resistance.
Section~4 presents a numerical example for our results, using sequences of networks sampled from a Lipschitz continuous graphon. Finally, conclusions and future work are exposed in Section~5.

\section{Graphs and Graphons}

This section contains the definition of graphons and some related notions and facts that will be needed later.

\subsection{Graphons: basic notations and norms}
We begin by  summarizing some definitions and results from
\cite{lovasz2012large, janson2013graphons,avella2018centrality} about kernels and graphons.
The space of all bounded symmetric measurable functions $W:[0,1]^2\rightarrow \mathbb{R}$ is denoted by $\mathcal{W}$. The elements of this space are called \textit{kernels}, because of their connection with integral operators, illustrated below. The set of all kernels $W \in \mathcal{W}$ such that $0 \leq W \leq 1$ is denoted by $\mathcal{W}_0$ and their elements are called \textit{graphons}, whose name is a contraction of graph-function. The set of all kernels $W \in \mathcal{W}$ such that $-1\leq W \leq 1$ is denoted by $\mathcal{W}_1$.
The degree function of a graphon is defined as:
$$
d(x):=\int_0^1W(x,y) \dy .
$$
We denote by $\delta_W$ the infimum of $d(x)$. 

Every function $W \in \mathcal{W}$ defines an integral operator $T_W : L^2[0,1]\rightarrow L^2[0,1]$ by:
$$
\left(T_Wf\right)(x):=\int_0^1W(x,y)f(y) \dy .
$$
If $W$ is continuous, then $T_W$ is also an operator  $T_W: C[0,1] \to C[0,1] $.

For $1\leq p <\infty$,  the $L^p$ norm of a kernel is
$$
\Vert W \Vert_p:=\left(\int_{[0,1]^2}\vert W(x,y)\vert^p \dx \dy\right)^{\!\! 1/p}
$$
and its \textit{cut norm} is
$$
\Vert W \Vert_\square:=\sup_{S,T\subseteq[0,1]}\left\vert\int_{S\times T}W(x,y) \dx \dy\right\vert .
$$
For $W\in\mathcal{W}_1$, we have the following inequalities between $L^p$ norms and the cut norm \cite[Equation 8.14]{lovasz2012large}:
\begin{equation} \label{eq:norms-graphons}
    \Vert W \Vert_\square \leq \Vert W \Vert_1\leq
\Vert W \Vert_2 \leq \Vert W \Vert_1^{1/2} \leq 1.
\end{equation}
By considering the operator $T_W$ associated to a kernel $W \in \mathcal W$, we can define the operator norm:
$$
\vertiii{T_W}:=\sup_{\substack{f \in L^2[0,1] \\ \Vert f \Vert_{2}=1 }}\Vert T_Wf\Vert_2 .
$$
For the elements of $\mathcal{W}_1$, the cut and operator norms are related by \cite[Equation 4.4 and Lemma E.6]{janson2013graphons}:
\begin{equation} \label{eq:norms-operators}
\Vert W \Vert_\square\leq\vertiii{T_W}\leq\sqrt{8}\Vert W \Vert_\square^{1/2}.    
\end{equation}

\subsection{Sampled Graphs} \label{sec:sampling}

A graphon  $W$ can be used to generate a random graph with  $N$ vertices, by using the following sampling method in two steps:

\textit{1. Complete Weighted Graph $\bar G_N$}: 
let $X=(X_1, \dots, X_N)$ be a sequence of independent random variables uniformly distributed on the interval $[0,1]$. We generate the complete weighted graph $\bar G_N$ with $N$ vertices, whose adjacency matrix is defined as: $\bar A_N(i,j)=W(X_{(i)},X_{(j)})$ for all $i,j$ in $\until{N}$, where $X_{(i)}$ is the $i$-th order statistic of the samples $X_1, \dots, X_N$. 

\textit{2. Simple Graph $G_N$}: from $\bar G_N$, we generate the simple graph $G_N$ with $N$ vertices by connecting each pair of distinct vertices $i\neq j$ with probability $\bar A_N(i,j)$ independently of the other edges.

The {\em degrees} of the vertices of $\bar G_N$ are denoted by $\bar d_i$ 
(i.e., $\bar d_i$ is the $i$th row-sum of $\bar A_N$)
and the normalized degrees by $\bar \delta_i=\bar d_i/N$.
We introduce also the diagonal degree matrix $\bar D_N=\mathrm{diag}[\bar d_1,\cdots,\bar d_N]$ and the Laplacian matrix $\bar L_N=\bar D_N-\bar A_N$. We denote the eigenvalues of $\bar L_N$ as $\bar \lambda_i\leq\cdots\leq\bar \lambda_N$ and its normalized eigenvalues as $\bar \mu_i=\bar \lambda_i/N$.

Similarly, we denote the degrees of $G_N$ by $d_i$ 
(i.e., $d_i$ is the $i$th row-sum of $A_N$)
and the normalized versions by $\delta_i=d_i/N$. The degree matrix is defined as $D_N=\mathrm{diag}[d_1,\cdots,d_N]$ and the Laplacian matrix as $L_N=D_N-A_N$. The eigenvalues of the Laplacian matrix are denoted by $\lambda_1\leq\cdots\leq\lambda_N$ and the normalized eigenvalues as $\mu_i=\lambda_i/N$.
Notice that $\bar A_N$ is the expectation of $A_N$ given $X$,
and hence $\bar d_i$ is the expectation of $d_i$ given $X$.

When needed, we will also use 
$\bar d_{(1)}\le \dots \le \bar d_{(N)}$
to denote degrees $\bar d_1, \dots, \bar d_N$ re-arranged in non-decreasing order, 
and similarly we will define  $\bar \delta_{(i)}$'s,
$d_{(i)}$'s and $\delta_{(i)}$'s with a non-decreasing reordering of the corresponding (normalized) degrees.

By considering a uniform partition of $[0,1]$ into the intervals $B_i^N$, where $B_i^N=[(i-1)/N,i/N)$ for $i=1,\ldots, N-1$ and $B_N^N=[(N-1)/N,1]$, we define the following step functions concerning degrees:
$$
 d_N(x)=\sum_{i=1}^N \delta_i\mathbbm{1}_{B_i^N}(x), \qquad
 \tilde d_N(x)=\sum_{i=1}^N \delta_{(i)} \mathbbm{1}_{B_i^N}(x),
$$
and the following step functions concerning Laplacian eigenvalues:
$$
\mu_N(x)=\sum_{i=1}^N \mu_i\mathbbm{1}_{B_i^N}(x),  \qquad
\mu_N^\pi(x)=\sum_{i=1}^N \mu_{\pi(i)}\mathbbm{1}_{B_i^N}(x) ,
$$
where $\mathbbm{1}_A(x)$ is the indicator function and $\pi \in S_N$, i.e., $\pi$ is a permutation of $1, \dots, N$.

\subsection{Step Graphons Associated with Sampled Graphs} 

Given a (possibly weighted) graph $G$ with $N$ vertices and with weighted adjacency matrix whose entries are $a_{ij} \in [0,1]$, the step graphon $W_G$ associated with $G$ is 
defined as
$$
 W_G(x,y):=\sum_{i=1}^N\sum_{j=1}^N a_{ij}\mathbbm{1}_{B_i^N}(x)\mathbbm{1}_{B_j^N}(y) 
$$
and the corresponding operator is
$$
(T_{W_G}f)(x):=\sum_{j=1}^N a_{ij}\int_{B_j^N}f(y) \dy \;\;\text{for any} \; x\in B_i^N .
$$
For a step graphon $W_G$ we have \cite[Equation 8.15]{lovasz2012large}:
\begin{equation} \label{eq:norms-stepgraphon}
\Vert  W_G \Vert_{1}\leq\sqrt{2N}\Vert  W_G \Vert_\square .    
\end{equation}

To prove our main results, we will also need the following lemma,
concerning the Frobenius norm $\Vert A \Vert_F$ of the adjacency matrix of a graph and the operator norm of the associted step graphon. 
\begin{lemma}\label{lem_Frobenius}
Let $W$ be a step graphon associated with a graph $G$ with $N$ vertices. Then:
$$
\Vert A \Vert_F \leq \sqrt[4]{2N^5}\vertiii{T_{ W_G}}^{1/2} .
$$
\end{lemma}

\begin{proof}
We consider the $L^2$ norm of the step graphon:
$$
\Vert W_G \Vert_{2}=\left(\int_0^1\int_0^1  W_G^2(x,y) \dx \dy\right)^{1/2} .
$$
We can see that $W_G^2(x,y)=W_G(x,y)W_G(x,y)$ is the product of two step functions with the same partition. Using the property $\mathbbm{1}_{A\cap B}=\mathbbm{1}_A \mathbbm{1}_B$, we obtain:
$$
W_G^2(x,y)=\sum_{i=1}^N\sum_{j=1}^N a^2_{ij} \mathbbm{1}_{B_i^N}(x) \mathbbm{1}_{B_j^N}(y) ,
$$
and hence:
\begin{multline*}
 \Vert W_G \Vert_{2}^2= \int_0^1\int_0^1\sum_{i=1}^N\sum_{j=1}^N a_{ij}^2 \mathbbm{1}_{B_i^N}(x) \mathbbm{1}_{B_j^N}(y) \dx \dy\\
=\sum_{i=1}^N\sum_{j=1}^N \int_{\frac{i-1}{N}}^{\frac{i}{N}}\int_{\frac{j-1}{N}}^{\frac{j}{N}} a_{ij}^2 \dx \dy=\dfrac{1}{N^2} \sum_{i=1}^N\sum_{j=1}^N a_{ij}^2
= \dfrac{\Vert A \Vert_F^2}{N^2} .   
\end{multline*}
This gives $\Vert A \Vert_F=N\Vert W_G \Vert_2$. Finally, using \eqref{eq:norms-graphons}, \eqref{eq:norms-operators}, \eqref{eq:norms-stepgraphon} implies
$
\Vert A \Vert_F\leq N\Vert W_G \Vert_1^{1/2}\leq \sqrt[4]{2N^5} \Vert W_G
\Vert_\square^{1/2}\leq\sqrt[4]{2N^5} \vertiii{T_{W_G}}^{1/2}.
$
\end{proof}

For the graphs $\bar G_N$ and $G_N$ sampled from a graphon $W$
as described in Section~\ref{sec:sampling},
we will denote the corresponding step graphons with the short-hand notations
$\bar W_N := W_{\bar G_N}$ and $W_N := W_{G_N}$.

\subsection{Graphs Sampled from Piecewise Lipschitz Graphons}
We shall restrict our analysis to a class of graphons that is wide enough to be relevant for the applications, but leads to a tractable analysis. We therefore consider the class of piecewise Lipschitz graphons, some properties of which we recall from \cite{avella2018centrality}.
\begin{definition}[Piecewise Lipschitz graphon]
Graphon $W$ is said to be \textit{piecewise Lipschitz} if there exists a constant $L$ and a sequence of non-overlapping intervals $I_k=[\alpha_{k-1},\alpha_k)$ defined by $0=\alpha_0 < \cdots <\alpha_{K+1}=1$, for a finite non-negative integer $K$ such that for any $k, \ell$, any set $I_{k\ell}=I_k \times I_\ell$ and pairs $(x_1,y_1) \textrm{ and } (x_2,y_2)\in I_{k\ell}$ we have that:
$$
\vert W(x_1,y_1)-W(x_2,y_2)\vert \leq L(\vert x_1-x_2 \vert+\vert y_1-y_2 \vert) . 
$$
If $K=0$, then the graphon is said to be \textit{Lipschitz}.
\end{definition}
Notice that when $W$ is a piecewise Lipschitz graphon, the degree function $d(x)$ is piecewise continuous, and hence  $\delta_W$ is its minimum, and not just its infimum.

\begin{definition}[Large enough $N$]\label{largeN}
Given a piecewise Lipschitz graphon $W$ and $\nu<e^{-1}$, $N$ is \textit{large enough} if $N$ satisfies the following conditions:
\begin{subequations}
\begin{equation}
\dfrac{2}{N}<\min_{k\in \{1,\ldots , K+1\}}(\alpha_k-\alpha_{k-1}) , \label{eq_condi1}
\end{equation}
\begin{equation}
\dfrac{1}{N}\log\left(\dfrac{2N}{\nu}\right)+\dfrac{1}{N}(2K+3L)<\max_{x}d_W(x) , \label{eq_condi2}
\end{equation}
\begin{equation}
Ne^{-N/5}<\nu . \label{eq_condi3}
\end{equation}
\end{subequations}
\end{definition}

The following  result is given in \cite{avella2018centrality} as Theorems~1 and~2.
\begin{lemma} \label{lemma1}
For a piecewise Lipschitz graphon $W$ and $N$ large enough, with probability at least $1-\nu$:
\begin{equation}
\vertiii{T_{\bar {W}_N}-T_W}\leq 2\sqrt{(L^2-K^2)b_N^2+Kb_N}=:\vartheta(N) , \label{bound_operator_H}
\end{equation}
\begin{equation}\label{bound_degree_H}
\Vert \bar d_N(x)-d(x) \Vert_2\leq \vartheta(N) ,
\end{equation}
and with probability at least $1-2\nu$:
\begin{align}
\vertiii{T_{W_N}-T_W}&\leq \sqrt{\dfrac{4\log (2N/\nu)}{N}}+\vartheta(N)=:\phi(N) , \label{bound_operator_G}
\end{align}
\begin{equation}\label{bound_degree_G}
\Vert d_N(x)-d(x) \Vert_2\leq \phi(N) ,
\end{equation}
where $b_N:=\frac{1}{N}+\sqrt{\frac{8\log(N/\nu)}{N+1}}$.
\end{lemma}

From Lemma~\ref{lemma1} we can see that the constant $\nu$ determines the probability with which the results hold, such that if we want a higher probability, the value of $N$ will increase to satisfy the \textit{large enough} condition. The constant $\nu$ will appear in most of the results of the paper.

To obtain the main results of our paper (see Section~\ref{sec:main}) we will consider graphons which are piecewise Lipschitz. Moreover, when needed, we will consider graphons which are bounded away from zero, i.e., whose infimum (denoted by $\eta_W$, and which is actually a minimum under the piecewise Lipschitz assumption) is strictly positive. Graphons which are bounded away from zero are also known as graphons having `minimal degree' \cite{avella2018centrality}, since the assumption $W(x,y) \ge \eta_W>0$ for all $x,y$ has the following implications about the degrees, both of the graphon itself and of the graphs sampled from the graphon:  $\delta_W \ge \eta_W >0$ and $\bar \delta_i \ge \eta_W$ for all $i = 1, \dots, N$.

\subsection{Laplacian Operator of a Graphon}

The Laplacian matrix of a graph is defined as the difference between the degree matrix and the adjacency matrix. In analogy with this definition, we can define a Laplacian operator for graphons $\mathcal{L}_W:L^2[0,1]\to L^2[0,1]$ as:
\begin{equation}\label{eq:laplacian-operator}
(\mathcal{L}_Wf)(x):=d(x)f(x)- (T_W f)(x).   
\end{equation}
If the graphon is continuous, $\mathcal{L}_W$ is also an operator in the space of continuous functions $\mathcal{L}_W:C[0,1]\to C[0,1]$, see \cite{von2008consistency}. 
The spectrum of this operator is composed by an essential spectrum located in the range of the degree function $d(x)$ and a finite number of isolated eigenvalues $\kappa_i$, which can only have accumulation points in the boundaries of the essential spectrum. The isolated eigenvalues are contained in the interval $[0,1]$ 
and $\kappa_1=0$ is always an eigenvalue with a constant eigenfunction associated $\psi_1(x)=k$.

\section{Main results on Laplacian spectrum} \label{sec:main}
This section contains our main results about the Laplacian spectrum, which regard the whole distribution of the eigenvalues (Section~\ref{sec:spectrum-dsitribution}), the spectral gap (Section~\ref{sect:gap}) and the average effective resistance (Section~\ref{sec:resistance}). We conclude the section with some remarks about an easy extension of our results to deterministically sampled graphs (Section~\ref{sect:deterministic}).

\subsection{Distribution of eigenvalues} \label{sec:spectrum-dsitribution}
For a large dense network, the distribution of the eigenvalues of the Laplacian matrix is close to the distribution of the degrees of the vertices \cite{hata2017localization}. Using results of perturbation theory, \cite{zhan2010distributions} derived a bound for the relative error in the estimation of the eigenvalues of the Laplacian matrix using the degrees of the network for simple graphs: 
\begin{equation} \label{eq:eigs-degrees-densegraphs}
  \dfrac{\Vert  \lambda_{G}- \tilde{d}_G \Vert_2}{\Vert  \tilde{d}_G \Vert_2}\leq\sqrt{\dfrac{N}{\Vert  \tilde{d}_G \Vert_1}} ,  
\end{equation}
where $\lambda_G$ is a vector with the Laplacian eigenvalues arranged in non-decreasing order and $\tilde d_G$ is a vector with the degrees of the network arranged in non-decreasing order.
In particular, for a sequence of graphs where at least a constant fraction of vertices have a degree growing linearly with $N$, the right-hand side of \eqref{eq:eigs-degrees-densegraphs} decays to zero as $O(1/\sqrt N)$.

Graphs $G_N$ sampled from a graphon as described in Section~\ref{sec:sampling} are dense graphs, and it is natural to look for an analogous of \eqref{eq:eigs-degrees-densegraphs}, so as to show that the eigenvalues of the Laplacian of $G_N$ are mostly determined by the {\em reordered degree function of the same graph}, with an error bounded by a quantity only depending on $N$ and on the graphon (see Proposition~\ref{prop_Distribution}).

\begin{proposition}\label{prop_Distribution}
For a piecewise Lipschitz graphon $W$ and $N$ large enough, with probability at least $1-2\nu$:
$$
\Vert \mu_N(x)-\tilde d_N(x)\Vert_{2}\leq \sqrt[4]{\dfrac{2}{N}}\sqrt{\vertiii{T_W}+\phi(N)},
$$
with $\phi(N)$ as in Lemma~\ref{lemma1}.
\end{proposition}

\begin{proof}
By definition:
\begin{multline*}
\Vert \mu_N(x)-\tilde d_N(x)\Vert_{2}^2=
\int_0^1\sum_{i=1}^N\vert\mu_i-\delta_{(i)}\vert^2 \mathbbm{1}_{B_i^N}(x) \dx\\
=\dfrac{1}{N}\sum_{i=1}^N\vert \mu_i-\delta_{(i)}\vert^2
=\dfrac{1}{N^3}\sum_{i=1}^N\vert \lambda_i-d_{(i)}\vert^2.
\end{multline*}
We can use the Wielandt-Hoffman Theorem \cite{horn2012matrix}, obtaining:
$$
\sum_{i=1}^N\vert\lambda_i-d_{(i)}\vert^2\leq\Vert A_N\Vert^2_F .
$$
and hence
$$
\Vert \mu_N(x)-\tilde d_N(x)\Vert_{2}\leq\dfrac{1}{N^{3/2}}\Vert A_N\Vert_F.
$$
By using Lemma~\ref{lem_Frobenius} we get:
$$
\Vert \mu_N(x)-\tilde d_N(x)\Vert_{2}\leq\dfrac{\sqrt[4]{2N^5}\vertiii{T_{W_N}}^{1/2}}{N^{3/2}}.
$$
Finally we notice that
$\vertiii{T_{W_N}} \le \vertiii{T_{W_N}-T_W} + \vertiii{T_{W}}$
and we use \eqref{bound_operator_G} from Lemma~\ref{lemma1} to obtain the desired result.
\end{proof}

Furthermore, we can approximate the distribution of the normalized Laplacian eigenvalues by using the {\em degree function of the graphon}, as follows.

\begin{proposition}\label{boundLaplacian}
For a piecewise Lipschitz 
graphon $W$ and $N$ large enough, 
with probability at least $1-2\nu$:
$$
\min_{\pi \in S_N} \Vert \mu^\pi_N(x)-d(x)\Vert_2\leq \sqrt[4]{\dfrac{2}{N}}\sqrt{\vertiii{T_W}+\phi(N)}+\phi(N),
$$
with $\phi(N)$ as in Lemma~\ref{lemma1}.
\end{proposition}
\begin{proof}
In addition to the step functions defined in Section~\ref{sec:sampling},
in this proof we will also use:
$$\bar \mu_N^\pi(x)=\sum_{i=1}^N \bar \mu_{\pi(i)}\mathbbm{1}_{B_i^N}(x) .$$

The goal of this proof is to show that, with probability at least $1-2\nu$,
there exists a permutation $\sigma$ such that:
$$
 \Vert \mu^\sigma_N(x)-d(x)\Vert_2\leq \sqrt[4]{\dfrac{2}{N}}\sqrt{\vertiii{T_W}+\phi(N)}+\phi(N).
$$
Notice that a different $\sigma$ might be used for different realizations of the random graph $G_N$.

By applying the triangle inequality in $\Vert \mu^\sigma_N(x)-d(x)\Vert_2$, we get that, for any~$\sigma$:
\begin{equation} \label{eq:proof-prop2}
\Vert \mu^\sigma_N(x)-d(x)\Vert_2 \leq \\
\Vert \mu^\sigma_N(x)- d_N(x)\Vert_2
+\Vert  d_N(x)-d(x) \Vert_2.
\end{equation}
For the first term, we have:
$$
\Vert \mu^\sigma_N(x)- d_N(x)\Vert_2=\dfrac{1}{N^{3/2}}\left(\sum_{i=1}^N\vert \lambda_{\sigma(i)}-  d_i\vert^2\right)^{1/2}.
$$
Then, we apply Wielandt-Hoffman theorem to $A_N = D_N - L_N$, which gives:
\begin{equation*}
\min_{\pi \in S_N} \sum_{i=1}^N\vert \lambda_{\pi(i)}-  d_i\vert^2
\leq \Vert  A_N \Vert_F^2 .
\end{equation*}
We choose $\sigma$ to be the permutation that achieves the above minimum,
so that we get:
$$
\Vert \mu^\sigma_N(x)- d_N(x)\Vert_2 
\le 
\frac{\Vert  A_N \Vert_F}{N^{3/2}}.
$$
For an upper bound on $\Vert  A_N \Vert_F$, we apply Lemma~\ref{lem_Frobenius}, and then we apply
the inequality 
$\vertiii{T_{W_N}} \le \vertiii{T_{W_N}-T_W} + \vertiii{T_{W}}$
and the bound \eqref{bound_operator_G} from Lemma~\ref{lemma1}.
We obtain that with probability at least $1-\nu$ there exists $\sigma$ such that:
\begin{equation}\label{eq_th1_2}
\Vert \mu^\sigma_N(x)- d_N(x)\Vert_2
\leq\sqrt[4]{\dfrac{2}{N}}\sqrt{\vertiii{T_W}+\phi(N)}.
\end{equation}
For the second term in the right-hand side of \eqref{eq:proof-prop2},
we use \eqref{bound_degree_G} from Lemma~\ref{lemma1}. 
Notice that Lemma~\ref{lemma1} ensures that with probability at least $1 -\nu$ both bounds \eqref{bound_operator_G} and \eqref{bound_degree_G} hold true, together;
this ensures that with the same probability both bounds \eqref{eq_th1_2} and \eqref{bound_degree_G} hold true, together.
\end{proof}

Proposition~\ref{boundLaplacian} is key in our analysis because it makes the connection between graphs and graphons: the result ensures that the degree function \emph{of the graphon} provides a good approximation of the eigenvalues \emph{of the graph}.  

The statement of Proposition~\ref{boundLaplacian} involves finding the best re-ordering $\pi$ of the Laplacian eigenvalues, so as to minimize $\Vert \mu^\pi_N(x)-d(x)\Vert_2$: if we want to know the function $\mu_N^\pi(x)$ that satisfies the corresponding upper bounds, it is necessary to evaluate $N!$ possible permutations.
A simpler statement can be obtained by adding a suitable monotonicity assumption
to a graphon bounded away from zero.  In the theorem below we will 
consider a graphon $W$ that is bounded away from zero and is non-decreasing, i.e., such that
$W(x_1,y) \le W(x_2,y)$ when $x_1 \le x_2$. A graphon being non-decreasing implies that its degree function is also non-decreasing, even though the converse is not true.

\begin{theorem}\label{thm:spectrum}
For a piecewise Lipschitz, non-decreasing graphon  with minimum $\eta_W>0$  
and for $N$ large enough, with probability at least $1-3 \nu$:
$$
\Vert \mu_N(x)-d(x)\Vert_2\leq \varphi(N)+\sqrt[4]{\dfrac{2}{N}}\sqrt{\vertiii{T_W}+\vartheta(N)}+\vartheta(N),
$$
with $\vartheta(N)$ as in Lemma~\ref{lemma1} and $\varphi(N)$ as in Lemma~\ref{concentration} below.
\end{theorem}

To prove Theorem~\ref{thm:spectrum}, we first need the following concentration results for the normalized degrees of $G_N$ and $\bar G_N$ and for the normalized eigenvalues of the corresponding Laplacian matrices.

\begin{lemma}\label{concentration}
Given a graphon $W$ with infimum $\eta_W>0$, if $N$ is large enough, with probability at least $1-\nu$ the normalized degrees  of the  graphs $G_N$ and $\bar G_N$ sampled from $W$ satisfy:
\begin{equation}\label{bound_delta}
\max_{i=1,\dots, N} \vert \delta_{(i)}-\bar \delta_{(i)}\vert \leq \sqrt{\dfrac{ \log(2N/\nu)}{N\eta_W}}:=\gamma(N),
\end{equation}
and with probability at least $1-2\nu$
the normalized eigenvalues of their Laplacian matrices
$L_N$ and $\bar L_N$ satisfy:
\begin{equation}\label{bound_mu}
\max_{i=1,\dots, N}  \vert \mu_i-\bar\mu_i \vert\leq 
\left(\frac{1}{\sqrt{\eta_W}}+2\right) \sqrt{\frac{\log(2N/\nu)}{N}}
:=\varphi(N).
\end{equation}
\end{lemma}

\begin{proof}
For the first part of the proof, we use Chernoff bound, as in \cite[Proof of Theorem 2]{chung2011spectra},
thanks to the remark that $\bar d_i$ is the expectation of $d_i$, conditioned on $X$, and that $d_i$ is the $i$th row-sum of $A_N$.
By the Chernoff bound, for any given~$i$:
$$
\mathrm{Pr}[\vert d_i-\bar d_i\vert>b\,\bar d_i]\leq \dfrac{\nu}{N} \quad \mathrm{if} \quad b\geq \sqrt{\dfrac{\log(2N/\nu)}{\bar d_i}}.
$$
Since $\bar d_i \ge N \eta_W$,
by considering $b=\sqrt{\dfrac{\log(2N/\nu)}{N\eta_W}}$,
for any given $i$, we have with probability at least $1-\nu/N$:
$$
\vert d_i-\bar d_i\vert\leq\sqrt{\dfrac{\log(2N/\nu)}{N\eta_W}}\bar d_i\leq\sqrt{\dfrac{\log(2N/\nu)}{N\eta_W}} \, \bar d_{(N)}.
$$
Hence, with probability at least $1-\nu$ this bound is true for all $i = 1,\dots, N$.
Since $D_N-\bar D_N$ is diagonal, 
\begin{equation}\label{lem3_1}
\Vert D_N-\bar D_N\Vert_2=\max_{i=1,\ldots,N}\vert d_i-\bar d_i\vert\leq\sqrt{\dfrac{\log(2N/\nu)}{N\eta_W}} \, \bar d_{(N)}
\end{equation}
with probability at least $1-\nu$.

From Weyl's Theorem,
$\max_i |d_{(i)} - \bar d_{(i)}| \le \Vert D_N-\bar D_N\Vert_2$,
which ends the proof of \eqref{bound_delta}, recalling that $\bar d_{(N)} \le N$, $d_{(i)}=N\delta_{(i)}$ and $\bar d_{(i)}=N\bar \delta_{(i)}$.

For the second part of the lemma,
we have:
\begin{align*}
\Vert L_N-\bar L_N \Vert_2&=\Vert D_N-A_N-\bar D_N+\bar A_N \Vert_2 \\
 &\leq\Vert D_N-\bar D_N\Vert_2+\Vert \bar A_N-A_N \Vert_2.
\end{align*}
From \cite[Theorem 1]{chung2011spectra}, we have that with probability at least $1-\nu$:
\begin{equation}\label{prob_event}
\Vert A_N-\bar A_N \Vert_2\leq\sqrt{4\bar d_{(N)}\log(2N/\nu)}.
\end{equation}
By combining \eqref{lem3_1} and \eqref{prob_event} we have with probability at least $1-2 \nu$:
$$
\Vert L_N-\bar L_N \Vert_2\leq\sqrt{\dfrac{\bar d_{(N)}^2 \log(2N/\nu)}{N\eta_W}}+\sqrt{4\bar d_{(N)} \log(2N/\nu)}.
$$
By using Weyl's Theorem and considering the normalized eigenvalues we get:
$$
\max_i \vert \mu_i-\bar\mu_i \vert\leq \sqrt{\dfrac{\bar d_{(N)}^2 \log(2N/\nu)}{N^3\eta_W}}+\sqrt{\dfrac{4\bar d_{(N)} \log(2N/\nu)}{N^2}}  .
$$
Finally, since $\bar d_{(N)}\leq N$, we get the desired result.
\end{proof}

The proof of Theorem~\ref{thm:spectrum} and the proof of Theorem~\ref{main_theorem} in the next section heavily rely both on Lemma~\ref{lemma1} and on Lemma~\ref{concentration}.
More precisely, each proof will require the simultaneous use of various bounds from such lemmas; since each bound holds with some probability, the following lemma is about the joint probability of the bounds of interest.
\begin{lemma} \label{lemma-rem_prob}
Given a piecewise Lipschitz graphon $W$ with infimum $\eta_W>0$, if $N$ is large enough, then:
\begin{itemize}
    \item with probability $1-3 \nu$, the three bounds 
    \eqref{bound_operator_H}, \eqref{bound_degree_H} and  \eqref{bound_mu} hold true;
    \item with probability $1-3 \nu$, the four bounds \eqref{bound_operator_G}, \eqref{bound_degree_G},
     \eqref{bound_delta}  and  \eqref{bound_mu} hold true.
\end{itemize}
\end{lemma}
\begin{proof}
The first statement is immediately obtained: Lemma~\ref{lemma1} ensures that with probability at least $1 -  \nu$ both \eqref{bound_operator_H} and \eqref{bound_degree_H} hold true, while Lemma~\ref{concentration} ensures that \eqref{bound_mu} holds true with probability at least $1 - 2 \nu$. Hence, the probability that all three bounds hold true is at least $1 - \nu - 2\nu = 1 - 3\nu$.

For the second statement, we need a closer look at the proofs of Lemmas~\ref{lemma1} and~\ref{concentration}.
The same event \eqref{lem3_1}, which has probability at least $1 - \nu$,
is used in the proof of both statements \eqref{bound_delta} and \eqref{bound_mu}
of Lemma~\ref{concentration}. 
Hence, the probability that both bounds hold together is at least $1-2 \nu$.
Moreover, the same event \eqref{prob_event}, 
which has probability at least $1-\nu$ thanks to \cite[Theorem 1]{chung2011spectra}, is used
both in the above proof of Lemma~\ref{concentration} and in the proof of the two statements \eqref{bound_operator_G} and \eqref{bound_degree_G} in Lemma~\ref{lemma1} 
(see \cite[Proof of Theorem 1]{avella2018centrality} for the latter proof).
Hence, the probability that all the bounds \eqref{bound_delta}, \eqref{bound_mu}, \eqref{bound_operator_G}, and \eqref{bound_degree_G}  hold true together is at least $1-3 \nu$.
\end{proof}

Now we have all the tools needed to prove Theorem~\ref{thm:spectrum}.
\begin{proof}[Proof of Theorem~\ref{thm:spectrum}]
In this proof, we will make use of the bounds \eqref{bound_operator_H}, \eqref{bound_degree_H} and  \eqref{bound_mu} from Lemmas~\ref{lemma1} and \ref{concentration}; by Lemma~\ref{lemma-rem_prob}, the event that all three bounds hold true has probability at least $1 - 3 \nu$.

We start from the left-hand side of the claimed inequality, and we use the triangle inequality to obtain that:
\begin{multline} \label{eq:proof-thm2}
\Vert \mu_N(x)-d(x)\Vert_2 \leq
\Vert \mu_N(x)-\bar \mu_N(x) \Vert_2\\
+\Vert \bar\mu_N(x)-\bar d_N(x)\Vert_2
+\Vert \bar d_N(x)-d(x) \Vert_2.
\end{multline}
The first term can be rewritten and bounded by using \eqref{bound_mu}:
\begin{equation}
\Vert \mu_N(x)-\bar \mu_N(x)\Vert_2=\left(\dfrac{1}{N}\sum_{i=1}^N\vert \mu_i-\bar \mu_i\vert^2\right)^{1/2}
\leq\varphi(N) \label{eq_th2_1}.
\end{equation}
For the second term,
we have:
\begin{equation*}
\Vert \bar\mu_N(x)-\bar d_N(x)\Vert_2 \\
 = \dfrac{1}{N^{3/2}}\left(\sum_{i=1}^N\vert \bar\lambda_i- \bar d_i\vert^2\right)^{1/2}.
\end{equation*}
By applying Wielandt-Hoffman theorem to $\bar A_N = \bar D_N - \bar L_N$, we obtain:
\begin{equation}
\min_{\pi \in S_N} \sum_{i=1}^N\vert \bar\lambda_{\pi(i)}- \bar d_i\vert^2
\leq \Vert \bar A_N \Vert_F^2 .
\label{eq:WH}
\end{equation}
The assumption that $W$ is non-decreasing implies that also
$\bar d_N(x)$ is non-decreasing, i.e., 
$\bar d_1 \le \dots \le \bar d_N$. 
Indeed, since $\lambda_1 \le \dots \le \lambda_N$ by definition
and $\bar d_1 \le \dots \le \bar d_N$ thanks to the monotonicity assumption,
the minimum in \eqref{eq:WH} (Wielandt-Hoffman theorem) is achieved for the identity permutation that leaves all positions unchanged. Then we obtain:
$$
    \Vert \bar\mu_N(x)-\bar d_N(x)\Vert_2 
    = 
    \dfrac{1}{N^{3/2}}\left(\sum_{i=1}^N\vert \bar\lambda_i- \bar d_i\vert^2\right)^{\!\! 1/2}
    \le 
    \frac{\Vert \bar A_N \Vert_F}{N^{3/2}}.
$$
Then, we can use  Lemma~\ref{lem_Frobenius}, and
the inequality 
$\vertiii{T_{\bar W_N}} \le \vertiii{T_{\bar W_N}-T_W} + \vertiii{T_{W}}$
with the bound \eqref{bound_operator_H} 
to obtain that 
\begin{equation}\label{eq_th2_2}
\Vert \bar\mu_N(x)-\bar d_N(x)\Vert_2
\leq\sqrt[4]{\dfrac{2}{N}}\sqrt{\vertiii{T_W}+\vartheta(N)}.
\end{equation}
For the third term in the right-hand side of \eqref{eq:proof-thm2},
we use \eqref{bound_degree_H}. 

Finally, the desired result is obtained from \eqref{eq:proof-thm2} by combining the three
bounds \eqref{eq_th2_1}, \eqref{eq_th2_2} and \eqref{bound_degree_H}.
\end{proof}

Theorem~\ref{thm:spectrum} has implications on the asymptotic behavior of
$\Vert \mu_N(x)-d(x)\Vert_2$, which are discussed in the next remark.

\begin{remark} \label{remark:asymptotic-thm1}
The upper bound on $\Vert \mu_N(x)-d(x)\Vert_2$ in Theorem~\ref{thm:spectrum}
holds true with probability at least $1 - 3 \nu$, 
and has an expression which depends both on $\nu$ and on $N$.
We are interested in its asymptotic behaviour for $N \to \infty$.
When we consider a constant $\nu$, we can easily see that 
this upper bound  goes to zero as $O((\log(N)/N)^{1/4})$,
since $\varphi(N) =O((\log(N)/N)^{1/2})$ and $\vartheta(N) = O((\log(N)/N)^{1/4})$;
moreover, if the graphon is Lipschitz ($K=0$),
then the upper bound goes to zero as $O((1/N)^{1/4})$, since in this case $\vartheta(N) = O((\log(N)/N)^{1/2})$.
It is interesting to notice that  all these asymptotic behaviours remain
the same also when we consider $\nu = 1/N^\alpha$ for any positive constant $\alpha$, since this only affects  constant factors.
By choosing $\alpha >1$, we can then apply Borel-Cantelli Lemma and
 obtain that, under the assumptions of Thm.~\ref{thm:spectrum}, almost surely
$\Vert \mu_N(x)-d(x)\Vert_2$ decays to zero
as  $O((\log(N)/N)^{1/4})$, 
and under the further assumption that the graphon is  Lipschitz then almost surely
$\Vert \mu_N(x)-d(x)\Vert_2$ goes to zero as  $O((1/N)^{1/4})$.  

Similarly, from Prop.~\ref{boundLaplacian} we  obtain that for any piecewise Lipschitz graphon
almost surely $\min_{\pi \in S_N} \Vert \mu^\pi_N(x)-d(x)\Vert_2$
 decays to zero as  $O((\log(N)/N)^{1/4})$, 
 and for any  Lipschitz graphon
almost surely $\min_{\pi \in S_N} \Vert \mu^\pi_N(x)-d(x)\Vert_2$
 decays to zero as  $O((1/N)^{1/4})$.
\end{remark}

\subsection{Spectral gap}\label{sect:gap}

The results in Section~\ref{sec:spectrum-dsitribution} concern the distribution of eigenvalues of the Laplacian matrix. However, it is often useful to obtain more detailed information on small eigenvalues, and in particular on the second largest, to see its distance from zero. This distance, also known as the spectral gap, is a measure of how well connected is the graph and plays a crucial role in shaping the properties of graph-based dynamics such as random walks on graphs and consensus-seeking systems~\cite{PD-DS:91,FRC:06,fagnani2018averaging}.

In this section we give some results on $\bar \mu_2$, the spectral gap of $\bar G_N$, and on $\mu_2$, the spectral gap of $G_N$: the results about $\bar \mu_2$ will become useful in the next section. The two spectral gaps are closely related, due to Lemma~\ref{concentration}. 
\begin{remark} \label{rem_mu2-barmu2}
Given a graphon $W$ with infimum $\eta_W>0$, by Lemma~\ref{concentration} we have
$|\mu_2 - \bar \mu_2| \le \varphi(N)$ with probability at least $1-2\nu$.
By taking $\nu = 1/N^{\alpha}$ for some $\alpha>1$, we obtain
that $|\mu_2 - \bar \mu_2| = O((\log(N)/N)^{1/2})$ with probability at least $1-O(1/N^\alpha)$, and hence we
can apply Borel-Cantelli Lemma to conclude that almost surely
$|\mu_2 - \bar \mu_2| $ decays to zero as $O((\log(N)/N)^{1/2})$.
\end{remark}

\begin{lemma}\label{lemma:second_eig}
For a complete weighted graph sampled from a graphon with infimum $\eta_W$:
$$
\bar\mu_2\ge\eta_W.
$$
\end{lemma}
\begin{proof}
We  use the variational characterization of eigenvalues (Courant-Fischer theorem).
Since $\bar \lambda_1 = 0$ with eigenvector $1_{\!N}$ (the all-ones vector of size $N$),
\begin{equation} \label{eq:courant-fischer}
\bar \mu_2  = \frac{1}{N} \bar \lambda_2
= \frac{1}{N} 
  \min_{\substack{x: \,  x^T \! x=1 \\  x^T\!1_N=0}} x^T \bar L_N x.
\end{equation}
Since $\bar L_N$ is a symmetric Laplacian matrix,
\[ x^T \bar L_N x 
\!=\! \frac{1}{2}\!\sum_i \sum_{j}  \bar A_N(i,j) (x_i-x_j)^2\!
\ge\! \frac{\eta_W}{2}\sum_i \sum_{j}\!  (x_i-x_j)^2 \!.
\]
Then notice that
$\sum_i \sum_{j}  (x_i-x_j)^2 =  2N[(\sum_i x_i^2) -(\sum_i x_i)^2 ]$
and hence for all $x$ such that $x^Tx=1$ and $x^T1_N = 0$ we have
$\sum_i \sum_{j}  (x_i-x_j)^2 = 2 N$, so that
$ x^T \bar L_N x \ge N \eta_W$.
With this, together with \eqref{eq:courant-fischer}, we can conclude that $\bar \mu_2 \ge \eta_W$.
\end{proof}

By combining \eqref{bound_mu} and Lemma~\ref{lemma:second_eig}, we obtain the following lower bound for $\mu_2$.
\begin{proposition}\label{prop_mu_2}
For a simple graph sampled from a graphon with infimum $\eta_W$ and for $N$ large enough, with probability at least $1-2\nu$:
$$
\mu_2\ge\eta_W-\varphi(N),
\quad \text{with $\varphi(N)$ as in \eqref{bound_mu}}. $$
\end{proposition}

From Proposition~\ref{prop_mu_2} we can see that a sufficient condition to guarantee that the spectral gap $\mu_2$ remains bounded away from zero a.s.\ 
is having a graphon with $\eta_W>0$.
Another case in which $\mu_2$  is guaranteed to remain bounded away from zero a.s.\ is given by Proposition~\ref{convergence} below
(based on the results in \cite{von2008consistency}): 
when the graphon $W$ is continuous, has $\delta_W >0$ and its zero eigenvalue $\kappa_1 = 0$ has multiplicity one.

\begin{proposition}\label{convergence}
Let $W$ be a continuous graphon.
Let $M$  be the number of isolated eigenvalues of $W$ in the interval $[0,\delta_W]$,  
counted with their multiplicities, and let $0 = \kappa_1 \le \kappa_2 \le \dots \le \kappa_M$ be such eigenvalues.
Define $\rho = \kappa_2$ if $M \ge 2$, and $\rho = \delta_W$ if $M = 1$.
Then,

$$
\lim_{N \to \infty} \mu_2 = \rho, \qquad \mathrm{a.s.}
$$

\end{proposition}

\begin{proof}
The results in \cite{von2008consistency} imply that $\lim_{N \to \infty} \bar \mu_2 = \rho$ a.s., as we show below.
Consider the operator $\mathcal{L}_{P_N}:C[0,1]\to C[0,1]$:
$$
(\mathcal{L}_{P_N} f)(x):=d_{P_N}(x)f(x)-\int_0^1 W(x,y)f(y) \, \mathrm dP_N(y),
$$
where $d_{P_N}(x):=\int_0^1W(x,y) \, \mathrm dP_N(y)$, $P_N:=1/N\sum_{i=1}^N \delta_{X_i}$ is the empirical distribution and $\delta_{X_i}$ is the Dirac measure. 
Isolated eigenvalues of $\mathcal{L}_{P_N}$ are also eigenvalues of $\bar L_N$ \cite[Proposition 22]{von2008consistency}. According to \cite[Proposition 23]{von2008consistency}, $\mathcal{L}_{P_N}$ converges compactly to $\mathcal{L}_W$ a.s., which implies the convergence of isolated parts of the spectrum and due to the upper-semicontinuity, the limits of convergent sequences are the isolated eigenvalues of $\mathcal{L}_W$. Additionally, \cite[Proposition 6]{von2008consistency} implies that for an isolated eigenvalue $\kappa$ with multiplicity $m$, there are $m$ sequences of eigenvalues that converge to $\kappa$.
If $\kappa_1=0$ has multiplicity 1 and there are no isolated eigenvalues in the interval $(0,\delta_W)$, then the second eigenvalue $\bar \mu_2$ will converge to $\delta_W$ according to \cite[Proposition 24]{von2008consistency}. 

Finally, by Remark~\ref{rem_mu2-barmu2} we have that almost surely 
$
\lim_{N\to\infty}\mu_2=\lim_{N\to\infty}\bar\mu_2, 
$
which completes the proof.
\end{proof}

Notice that Proposition~\ref{convergence} requires the additional assumption that the graphon is continuous, which was not required in Proposition~\ref{prop_mu_2}, but on the other hand it does not require the graphon to be bounded away from zero and moreover it gives a much richer result, since it characterizes the almost sure limit of $\mu_2$ for $N \to \infty$.

\subsection{Average effective resistance} \label{sec:resistance}

We consider the simple graph $G_N$ as an electrical network where all the edges have resistance equal to 1. Between two vertices $i$ and $j$, we denote the effective resistance $R_{\mathrm{eff}}(i,j)$ as the electrical potential difference induced between $i$ and $j$ by a unit current injected in $i$ and extracted from $j$. 
The average effective resistance of $G_N$ is defined as:
$$
R^{\mathrm{ave}}_N:=\dfrac{1}{2N^2} \sum_{i = 1}^N \sum_{j=1}^N R_{\mathrm{eff}}(i,j).
$$    
This quantity is also related to the spectrum of the Laplacian matrix of the graph \cite{fagnani2018averaging}:
$$
R^{\mathrm{ave}}_N=\dfrac{1}{N}\sum_{i=2}^N\dfrac{1}{\lambda_i} .
$$

This handy characterization immediately leads to draw some conclusions about its asymptotic behavior.
\begin{remark}\label{remark_rateconv}
To find the asymptotic behavior for $N\to\infty$ of the average effective resistance of graphs sampled from a piecewise Lipschitz graphon bounded away from zero, we can see that:
$$
NR_N^{\mathrm{ave}}=\sum_{i=2}^N\dfrac{1}{\lambda_i}\leq(N-1)\dfrac{1}{\lambda_2}\leq\dfrac{1}{\mu_2}.
$$ 
As noticed in Remark~\ref{rem_mu2-barmu2}, 
Lemma~\ref{concentration} implies that
$\vert \mu_2- \bar\mu_2\vert$ goes to zero a.s.
Moreover, by Lemma~\ref{lemma:second_eig}, $\bar \mu_2 \ge \eta_W>0$.
From this, we can conclude that almost surely also $\mu_2$ remains bounded away from zero, and hence
$R_N^{\mathrm{ave}}=O(1/N)$ a.s.
Also, since $\lambda_N\leq N$ we get:
$$
NR_N^{\mathrm{ave}}\geq (N-1)\dfrac{1}{N}.
$$
Therefore $R_N^{\mathrm{ave}}=\Theta(1/N)$ a.s. 
\end{remark}
Considering that the distribution of the eigenvalues of the Laplacian matrix is similar to the distribution of the degrees, we can estimate the average effective resistance $R_N$ of a simple graph $G_N$ through the degree function of the graphon $W$, by defining an analogous quantity as: 
$$
R_{W,N}^{\mathrm{ave}}:=\frac{1}{N}\int_0^1 \dfrac{1}{d(x)} \dx.
$$

\begin{theorem}\label{main_theorem}
For a piecewise Lipschitz graphon $W$ with minimum $\eta_W>0$ and for $N$ satisfying conditions \eqref{eq_condi1}, \eqref{eq_condi2}, \eqref{eq_condi3} and condition:
\begin{equation}\label{eq_condi4}
\frac{\log(2N/\nu)}{N} < \frac{\eta_W^2}{1+2\eta_W}  ,
\end{equation}
let $R^{\mathrm{ave}}_N$ be the average effective resistance of a graph $G_N$ sampled from $W$. Then,
with probability at least $1-3\nu$:
\begin{multline*}
\left\vert R^{\mathrm{ave}}_N-R_{W,N}^{\mathrm{ave}}\right\vert\leq\\
\dfrac{1}{N(\eta_W-\gamma(N))}\left(\dfrac{1}{N}+\dfrac{\phi(N)}{\delta_W}+\dfrac{\sqrt[4]{2}\sqrt{\vertiii{T_W}+\phi(N)}}{N^{1/4}(\eta_W-\varphi(N))}\right),
\end{multline*}
with
$\phi(N)$ as in Lemma~\ref{lemma1} and $\varphi(N)$ and $\gamma(N)$ as in Lemma~\ref{concentration}.
\end{theorem}

\begin{proof}
In this proof, we will make use of the bounds \eqref{bound_operator_G}, \eqref{bound_degree_G},
\eqref{bound_delta} and  \eqref{bound_mu} from Lemmas~\ref{lemma1} and \ref{concentration}; by Lemma~\ref{lemma-rem_prob}, the event that all four bounds hold true has probability at least $1 - 3 \nu$.

If we define the step function
$$ r_N(x)= \sum_{i=2}^N\dfrac{1}{\lambda_i}\mathbbm{1}_{B_i^N}(x),$$ 
it is easy to see that
$
R^{\mathrm{ave}}_N=\Vert r_N(x) \Vert_{1}$ and $ R_{W,N}^{\mathrm{ave}}=\left\Vert \dfrac{1}{Nd(x)} \right\Vert_{1},
$
so that we have:
\begin{multline}
\left\vert R^{\mathrm{ave}}_N-R_{W,N}^{\mathrm{ave}}\right\vert=\left\vert \Vert r_N(x)\Vert_1-\left\Vert \dfrac{1}{Nd(x)}\right\Vert_1\right\vert \\
\leq
\left\vert\Vert r_N(x)\Vert_1 \! -
\!\left\Vert\dfrac{1}{Nd_N(x)}\right\Vert_1\right\vert
+
\left\vert\left\Vert\dfrac{1}{Nd_N(x)}\right\Vert_1
\!\!\!-\! \left\Vert\dfrac{1}{Nd(x)}\right\Vert_{1}\right\vert \!. \label{eq:proof-main-thm}
\end{multline}
We start by studying  the first term in \eqref{eq:proof-main-thm}.
We notice that 
$$ 
\left\Vert \dfrac{1}{N d_N(x)}\right\Vert_1
=  \frac{1}{N} \sum_{i=1}^N \frac{1}{d_{i}}
=  \frac{1}{N} \sum_{i=1}^N \frac{1}{d_{(i)}}    
= \left\Vert \dfrac{1}{N \tilde d_N(x)}\right\Vert_1 .
$$
We use this remark and the
reverse triangle inequality (i.e., $\Big\vert \Vert x \Vert-\Vert y \Vert \Big\vert\leq\Vert x-y\Vert$), to obtain
$$
\left\vert\Vert r_N(x)\Vert_1-\left\Vert\dfrac{1}{Nd_N(x)}\right\Vert_1\right\vert
\le
\left\Vert r_N(x)-\dfrac{1}{N\tilde d_N(x)}\right\Vert_{1} .
$$
Then, we have:
\begin{align*}
\left\Vert r_N(x)-\dfrac{1}{N\tilde d_N(x)}\right\Vert_{1}
&=\dfrac{1}{N^2\delta_{(1)}}+
    \dfrac{1}{N}\sum_{i=2}^N
    \left\vert \dfrac{d_{(i)}-\lambda_i}{d_{(i)}\lambda_i}\right\vert\\
&\leq \dfrac{1}{N^2\delta_{(1)}}+
    \dfrac{\sum_{i=2}^N\vert d_{(i)}-\lambda_i\vert}{N^3\mu_2\delta_{(1)}}.
\end{align*}
By using the norm inequality $\Vert \cdot \Vert_1 \leq \sqrt{N}\Vert \cdot \Vert_2$ and applying 
Wielandt-Hoffman Theorem to $A_N = D_N - L_N$, we obtain:
$$
\sum_{i=2}^N\vert d_{(i)}-\lambda_i\vert
\le 
\sqrt{N}\left(\sum_{i=2}^N\vert d_{(i)}-\lambda_i\vert^2\right)^{\!\!1/2}
\le 
\sqrt{N}\Vert A_N \Vert_F .
$$
Finally, we can use Lemma~\ref{lem_Frobenius} and \eqref{bound_operator_G} to obtain:
$$
\left\Vert r_N(x)-\dfrac{1}{N\tilde d_N(x)}\right\Vert_{1}\leq\dfrac{1}{N^2\delta_{(1)}}+\dfrac{\sqrt[4]{2}\sqrt{\vertiii{T_W}+\phi(N)}}{N^{5/4}\mu_2\delta_{(1)}}.
$$
Now we study the second term in \eqref{eq:proof-main-thm}. 
We use again the reverse triangle inequality and H\"older's inequality so that we have:
$$
\left\vert\left\Vert\dfrac{1}{Nd_N(x)}\right\Vert_1-\left\Vert\dfrac{1}{Nd(x)}\right\Vert_{1}\right\vert
\le
\left\Vert\dfrac{1}{Nd_N(x)}-\dfrac{1}{Nd(x)}\right\Vert_{2} \,.
$$
Then we obtain:
\begin{multline*}
\left\Vert\dfrac{1}{Nd_N(x)}-\dfrac{1}{Nd(x)}\right\Vert_{2}=\left(\int_0^1\left\vert \dfrac{d(x)-d_N(x)}{Nd(x)d_N(x)}\right\vert^2 \dx\right)^{\! 1/2}\\
\leq\dfrac{1}{N\delta_{(1)}\delta_W}\left(\int_0^1\left\vert d(x)-d_N(x)\right\vert^2 \dx\right)^{\! 1/2}.
\end{multline*}
Since $\left(\int_0^1\left\vert d(x)-d_N(x)\right\vert^2 \dx\right)^{\! 1/2}=\Vert d_N(x)-d(x) \Vert_{2}$, we can apply \eqref{bound_delta} and get:  
$$
\left\vert\left\Vert\dfrac{1}{Nd_N(x)}\right\Vert_1-\left\Vert\dfrac{1}{Nd(x)}\right\Vert_{1}\right\vert\leq\dfrac{1}{N\delta_{(1)}\delta_W}\phi(N).
$$
Using the bounds obtained for the two terms in \eqref{eq:proof-main-thm}, we get:
\begin{equation*}
\left\vert R^{\mathrm{ave}}_N-R_{W,N}^{\mathrm{ave}}\right\vert\!\leq\!\dfrac{1}{N^2\delta_{(1)}}+\dfrac{\sqrt[4]{2}\sqrt{\vertiii{T_W}+\phi(N)}}{N^{5/4}\mu_2\delta_{(1)}}+\dfrac{\phi(N)}{N\delta_{(1)}\delta_W}.
\end{equation*}
By using \eqref{bound_delta} and \eqref{bound_mu} 
we obtain: 
\begin{align*}
\left\vert R^{\mathrm{ave}}_N-R_{W,N}^{\mathrm{ave}}\right\vert
&\leq\dfrac{1}{N^2(\bar \delta_{(1)}-\gamma(N))}
+\dfrac{\phi(N)}{N\delta_W(\bar \delta_{(1)}-\gamma(N))}\\
&+\dfrac{\sqrt[4]{2}\sqrt{\vertiii{T_W}+\phi(N)}}{N^{5/4}(\bar \delta_{(1)}-\gamma(N))(\bar \mu_2-\varphi(N))}.
\end{align*}
Finally we get the desired result by using $\delta_{(1)}\geq\eta_W$ and
Lemma~\ref{lemma:second_eig}.
Notice that 
assumption $\eta_W>0$ and 
condition \eqref{eq_condi4} 
ensure that the denominators appearing in the upper bound are positive.
\end{proof}

\begin{remark}\label{remark:thm2}
Theorem~\ref{main_theorem} gives an upper bound on the absolute error
$\left\vert R^{\mathrm{ave}}_N-R_{W,N}^{\mathrm{ave}}\right\vert$.
This  bound
holds true with probability at least $1 - 3 \nu$, and has an expression which depends both on $\nu$ and on $N$.
We are interested in its asymptotic behaviour for $N \to \infty$.
When we consider a constant $\nu$ or $\nu=1/N^\alpha$,  this upper bound  goes to zero as $O((\log(N)/N^5)^{1/4})$,
and with the further assumption that the graphon is Lipschitz ($K=0$), it decays as
$O((1/N)^{5/4})$. 
By choosing $\alpha >1$, we can then apply Borel-Cantelli Lemma, and obtain that almost surely $\left\vert R^{\mathrm{ave}}_N-R_{W,N}^{\mathrm{ave}}\right\vert$ decays to zero as
$O((\log(N)/N^5)^{1/4})$, and moreover as $O((1/N)^{5/4})$ in case the graphon is Lipschitz.
It is also interesting to study the relative error
$\left\vert R^{\mathrm{ave}}_N-R_{W,N}^{\mathrm{ave}}\right\vert/R^{\mathrm{ave}}_N$.
Recalling  Remark~\ref{remark_rateconv} about the asymptotic behaviour of the denominator,
we obtain that the relative error almost surely decays to zero as $O((\log(N)/N)^{1/4})$,
and as $O((1/N)^{1/4})$ if the graphon is Lipschitz.
\end{remark}

\subsection{Deterministic Sampling}\label{sect:deterministic}

An alternative procedure for the generation of complete weighted graphs from graphons is the use of deterministic latent variables proposed in \cite{avella2018centrality}, such that the adjacency matrix of $\bar G_N$ is generated as:
$$
\bar A_N(i,j)=W(i/N,j/N) \qquad \text{for all } i,j\in\{1,\ldots,N\}.
$$
All the results of Section 3 (with exception of Proposition~\ref{convergence}) easily extend to deterministic sampling with minor adjustments, which we detail here.
Lemma~\ref{concentration} and Lemma~\ref{lemma:second_eig} do not depend on the sampling method and remain the same. As indicated in  \cite{avella2018centrality}, the factor $b_N$ of $\vartheta(N)$ and $\phi(N)$ in Lemma~\ref{lemma1} is redefined as $b_N:=1/N$ and \eqref{bound_operator_H} and \eqref{bound_degree_H} hold with probability 1 while \eqref{bound_operator_G} and \eqref{bound_degree_G} hold with probability at least $1-\nu$. 
With the new definition of $b_N$, Propositions~\ref{prop_Distribution} and \ref{boundLaplacian} hold with probability at least $1-\nu$ and Lemma~\ref{lemma-rem_prob} and Theorems~\ref{thm:spectrum} and \ref{main_theorem} hold with probability at least $1-2\nu$. The rates of convergence in Remarks~\ref{rem_mu2-barmu2} and \ref{remark_rateconv} do not change while for Remarks \ref{remark:asymptotic-thm1} and \ref{remark:thm2} the rate of convergence for piecewise Lipschitz and Lipschitz graphons is $O(1/N^{1/4})$.

Proposition~\ref{convergence} instead cannot be easily extended to the deterministic case because in \cite{von2008consistency} the compact convergence of the operators is proved by using Glivenko-Cantelli Theorem in one of the steps, which is formulated for random variables.

\section{Numerical example}

We consider the graphon $W(x,y)=1-0.8xy$,
which is  Lipschitz and bounded away from zero  (its minimum is $\eta_W=0.2$). 
Its  degree function is $d(x)=1-0.4x$, whose minimum is  $\delta_W = 0.6$. 
To validate the results, we consider a sequence of simple graphs that are randomly sampled from $W$ for $10\leq N \leq 1000$. Fig.~\ref{fig_norm2} presents the approximation of the distribution of the normalized eigenvalues by using the degree function of the graphon.

\begin{figure}
\centering
\includegraphics[width=75mm]{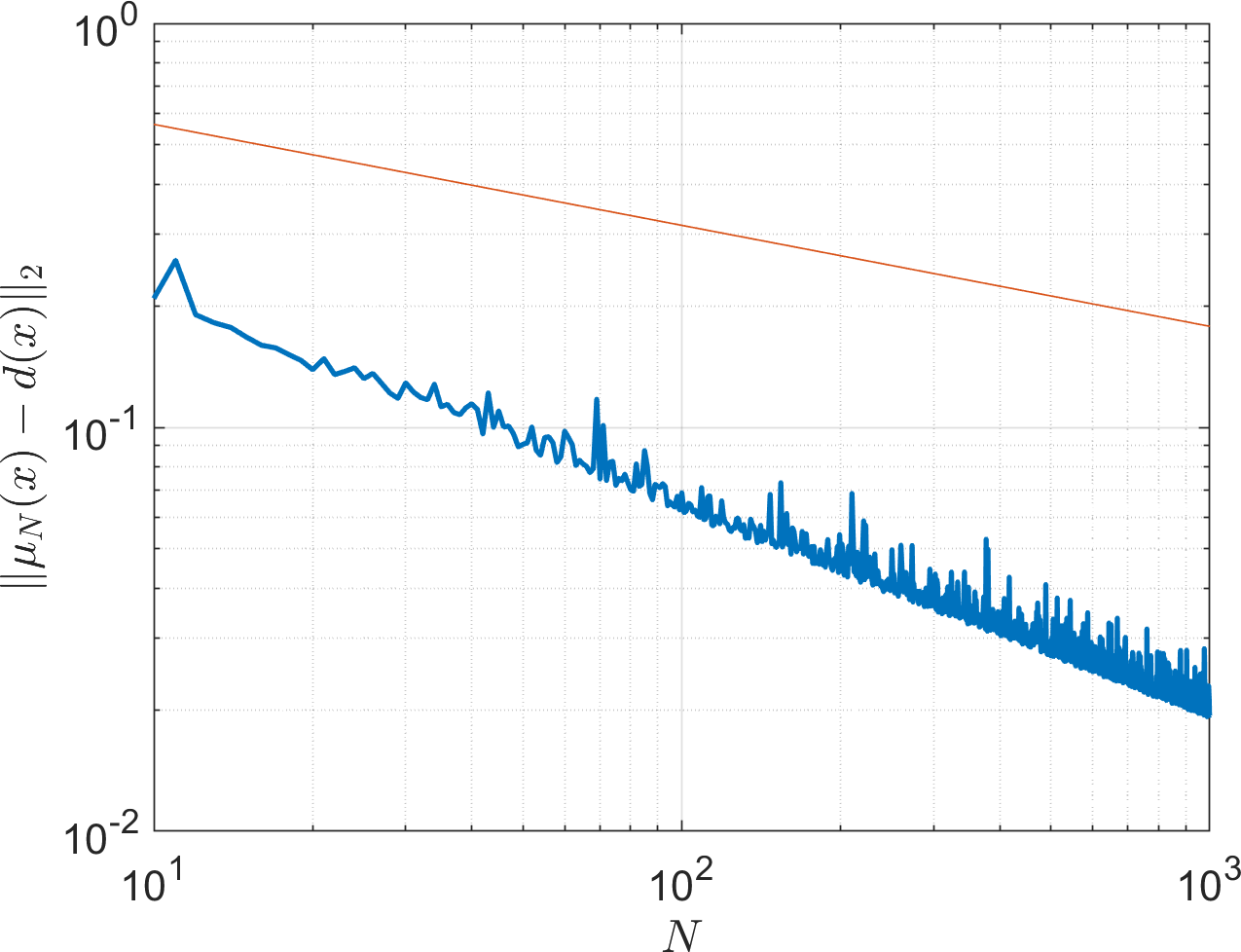} 
\caption{$\Vert \mu_N(x)-d(x)\Vert_2$ for growing $N$. Its decay is consistent with the upper bound $O(N^{-1/4})$, depicted in red. } 
\label{fig_norm2}
\end{figure}

By solving the eigenfunction equation (i.e., $(\mathcal{L}_W\psi)(x)=\kappa \psi(x)$) we find that the operator only has the trivial eigenvalue $\kappa_1=0$, so that, by Proposition~\ref{convergence} and Remark~\ref{rem_mu2-barmu2}, a.s. $\lim \bar \mu_2= \lim \mu_2 = \delta_W$. 
Fig.~\ref{fig_diff_second} shows the difference between the second normalized eigenvalues of $G_N$ and $\bar G_N$ while Fig.~\ref{fig_second} illustrates their convergence towards $\delta_W$ as per Proposition~\ref{convergence}. We can observe that the convergence of $\mu_2$ is slower than the convergence of $\bar \mu_2$.

\begin{figure}
\centering
\includegraphics[width=75mm]{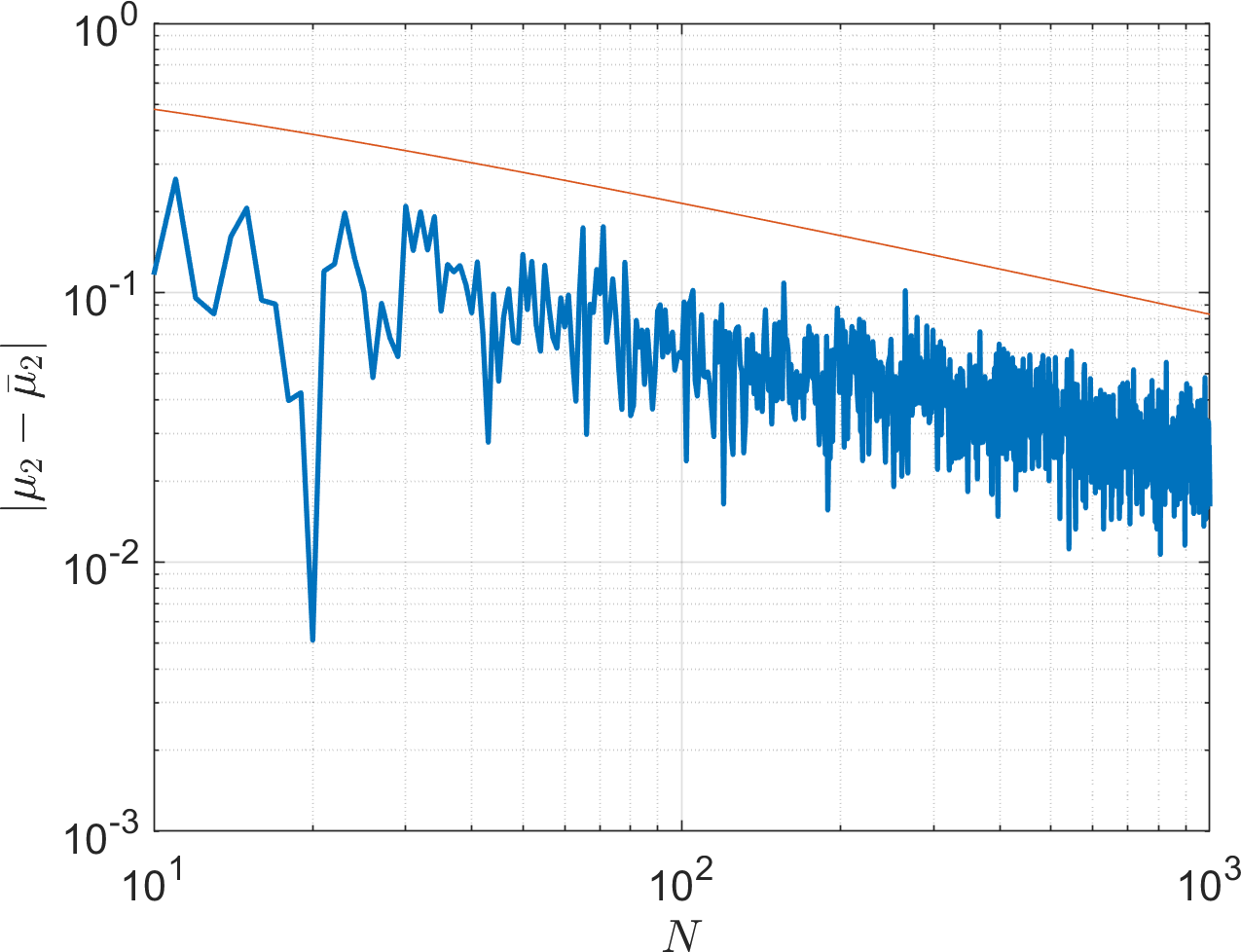} 
\caption{$\vert \mu_2-\bar\mu_2\vert$ for growing $N$. Its decay is comparable with the upper bound $(\log(N)/N)^{1/2}$, depicted in red. } 
\label{fig_diff_second}
\end{figure}

\begin{figure}
\centering
\includegraphics[width=75mm]{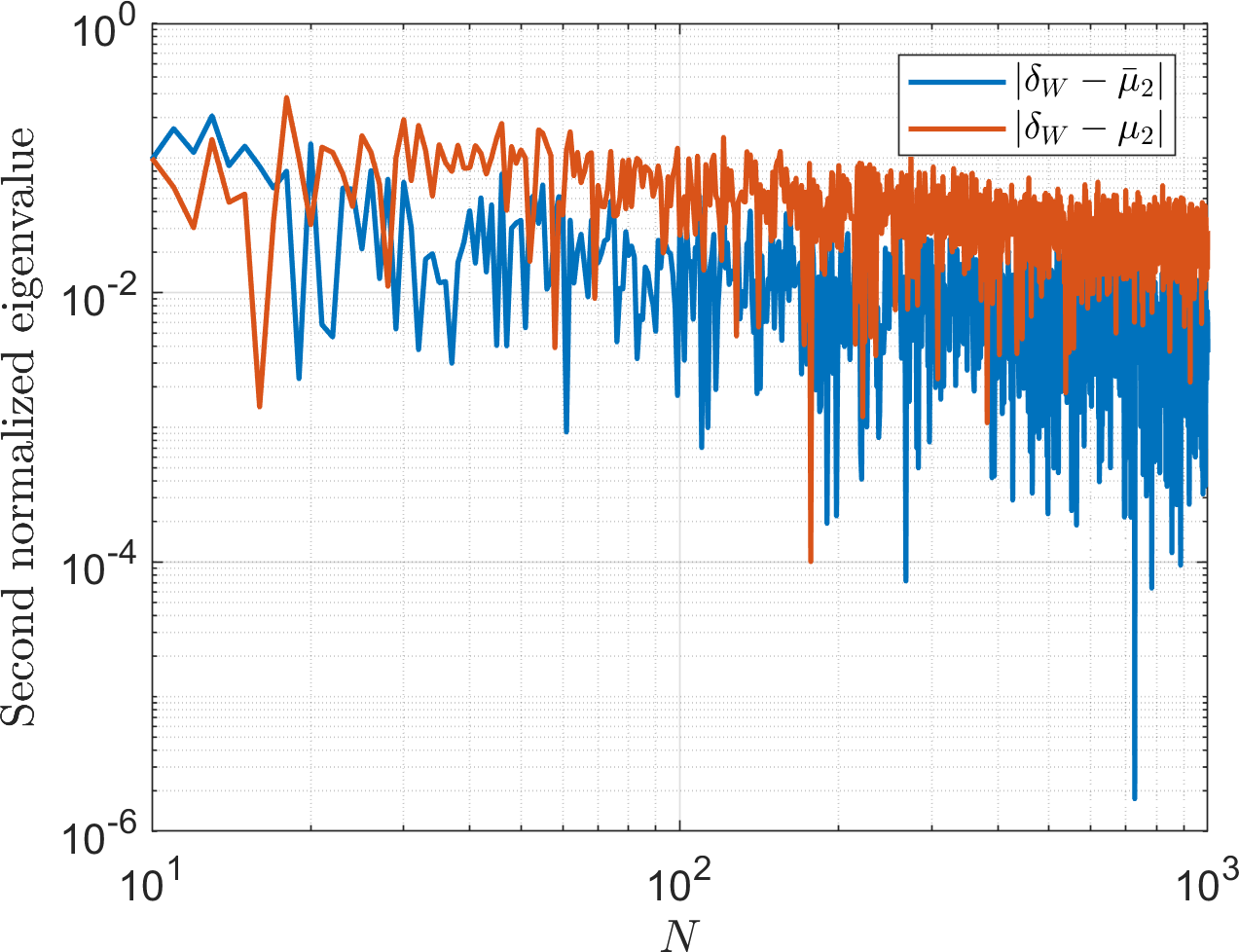} 
\caption{Convergence of second normalized eigenvalues for growing $N$.} 
\label{fig_second}
\end{figure}

The approximation of the average effective resistance is performed through:
$$
R_{W,N}^{\mathrm{ave}}=\dfrac{1}{N}\int_0^1 \dfrac{\dx}{1-0.4x}=-\dfrac{5}{2N}\log(0.6).
$$
Fig.~\ref{fig_average_relative} shows the approximation of the average effective resistance: the relative error plot suggests our bound on the convergence rate to be tight in its dependence on $N$.

\begin{figure*}
\centering
\includegraphics[width=75mm]{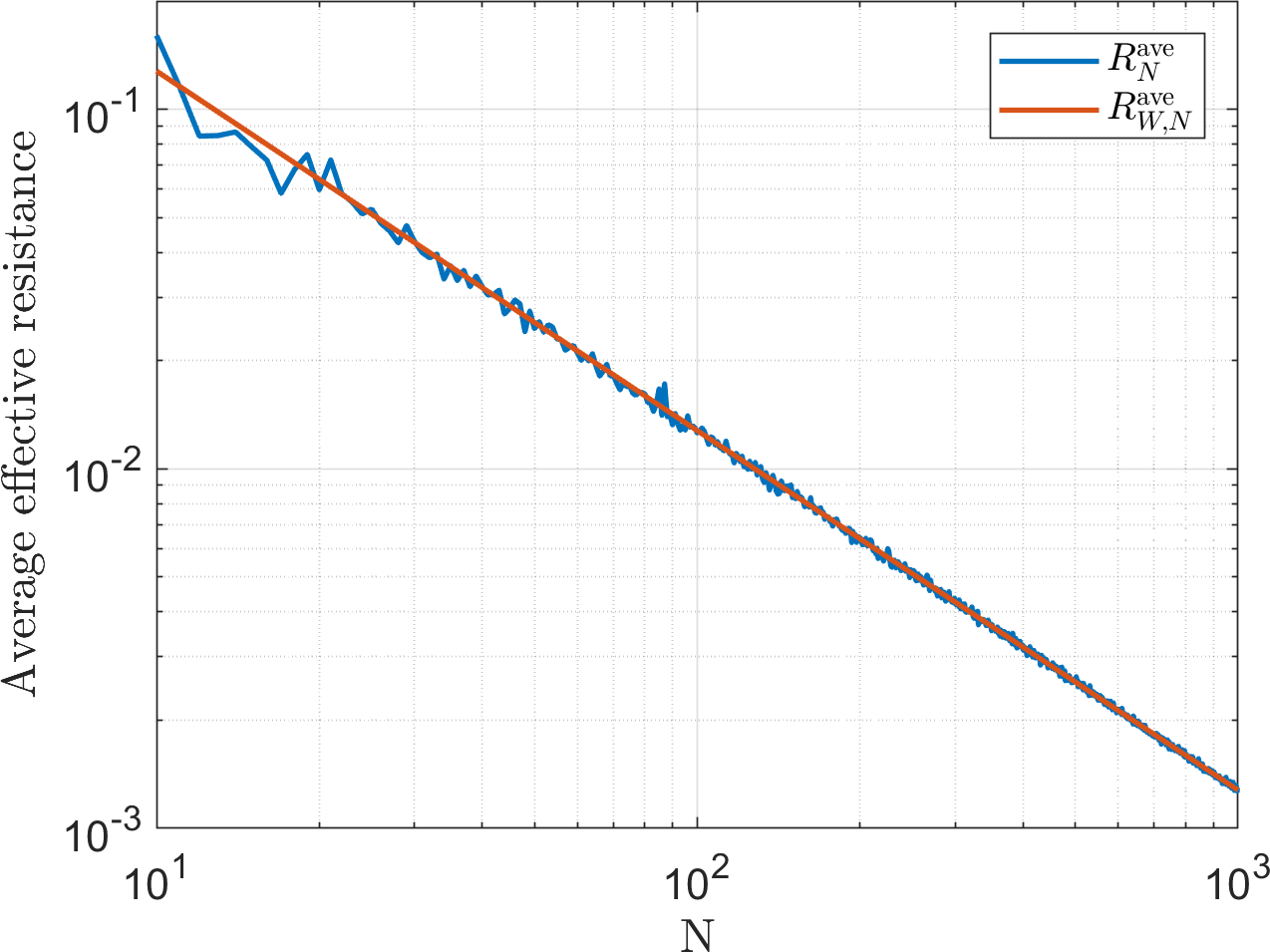} \qquad
\includegraphics[width=75mm]{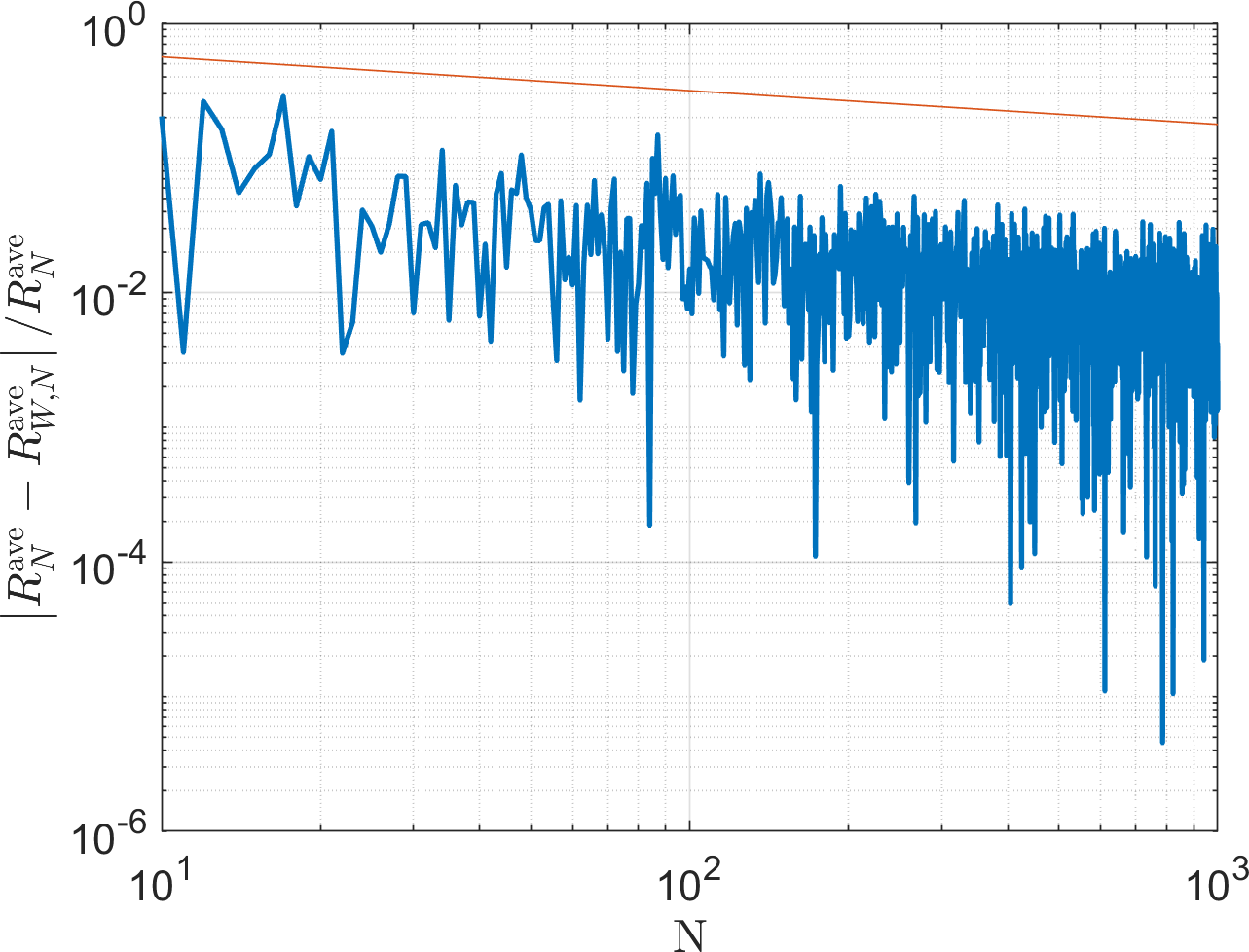} 
\caption{Left: Average effective resistance for growing $N$. Right: relative error $\left\vert R^{\mathrm{ave}}_N-R_{W,N}^{\mathrm{ave}}\right\vert/R^{\mathrm{ave}}_N$ for growing $N$. Its decay is comparable with $N^{-1/4}$, depicted in red.}
\label{fig_average_relative}
\end{figure*}

\section{Conclusion and future work}
In this paper, the spectrum of the Laplacian matrix of a network sampled from a graphon was analyzed using the degree function of the graphon. First, we showed that for networks derived from a graphon, the distribution of the eigenvalues of the Laplacian matrix is determined mainly by the degrees of the network. Then, we showed that the average effective resistance of a graph sampled from a graphon can be estimated by using the degree function of the graphon. For both problems, we have derived explicit bounds on the approximation error.

Even if this paper has shown some initial applications of the graphon Laplacian operator, numerous related questions remain open.
Indeed, our methods can be applied to estimate other functions of the Laplacian spectrum, such as the spectral zeta function \cite{karlsson2020spectral} and several performance metrics in estimation and control problems over networks \cite{garin2010survey,fagnani2018averaging}. Furthermore, the graphon Laplacian operator could be used to define suitable infinite-dimensional dynamical systems that approximate dynamical systems on finite-dimensional graphs, as done in~\cite{petit2019random} with the normalized Laplacian and in~\cite{gao2019graphon,delmas2020infinite} with the adjacency matrix. 
This line of work entails some technical difficulties, such as the lack of compactness of the graphon Laplacian operators.

Finally, we ought to recall that graphons are limited to approximate dense graphs, whereas many relevant networks are {\em sparse}. It is therefore an open question to develop the suitable tools to address these cases.

\bibliographystyle{IEEEtran}
\bibliography{ifacconf}

\begin{IEEEbiography}[{\includegraphics[width=1in,height=1.25in,clip,keepaspectratio]{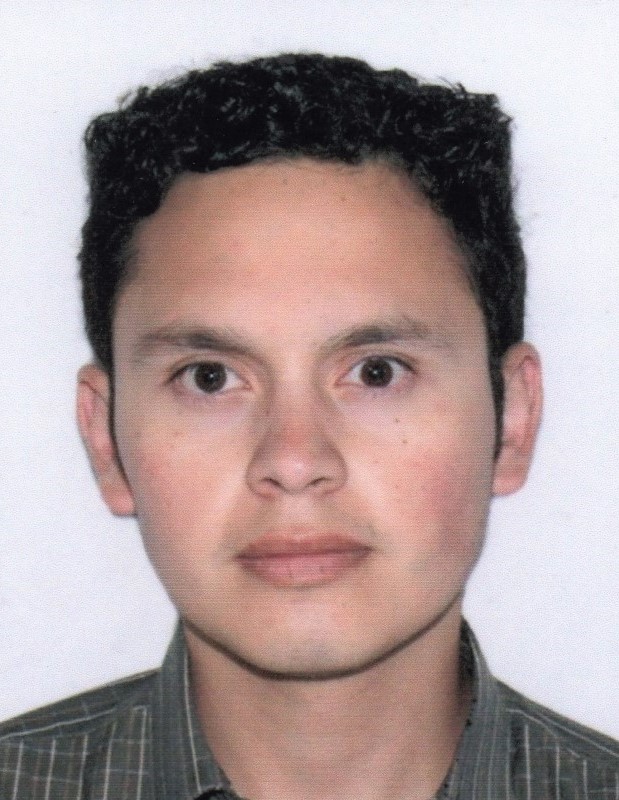}}]{Renato Vizuete}
received the B.S. degree (summa cum laude) in Electronics and Control Engineering from Escuela
Polit\'ecnica Nacional, Ecuador and the M.S. degree (tr\`es bien) in Systems, Control and Information Technologies
from Universit\'e Grenoble Alpes, France. He was the recipient of the Persyval-Lab Excellence Master Scholarship
from Universit\'e Grenoble Alpes in 2018. He is currently a PhD student at L2S CentraleSup\'elec and GIPSA-lab, France. His research interests include: control theory,
multi-agent systems, hybrid systems and networked control systems.
\end{IEEEbiography}

\begin{IEEEbiography}[{\includegraphics[width=1in,height=1.25in,clip,keepaspectratio]{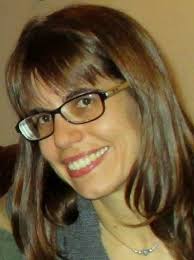}}]{Federica Garin}
(M’16) is  a  researcher  with  the NeCS  team  at  INRIA  and  GIPSA-lab,  Grenoble (France).  She  received  her  B.S.,  M.S.,  and  Ph.D. degrees in Applied Mathematics from Politecnico di Torino (Italy) in 2002, 2004, and 2008, respectively. She  was  a  post-doctoral  researcher  at  Universit\`a  di Padova  (Italy)  in  2008  and  2009,  and  at  INRIA Grenoble (France) in 2010.  She  is  an  Associate  Editor  in  the  IEEE-CSS Conference  Editorial  Board  and  in  the  European Control  Association  (EUCA)  Conference  Editorial Board. Her current research interests are in distributed algorithms and network control systems.
\end{IEEEbiography}

\begin{IEEEbiography}[{\includegraphics[width=1in,height=1.25in,clip,keepaspectratio]{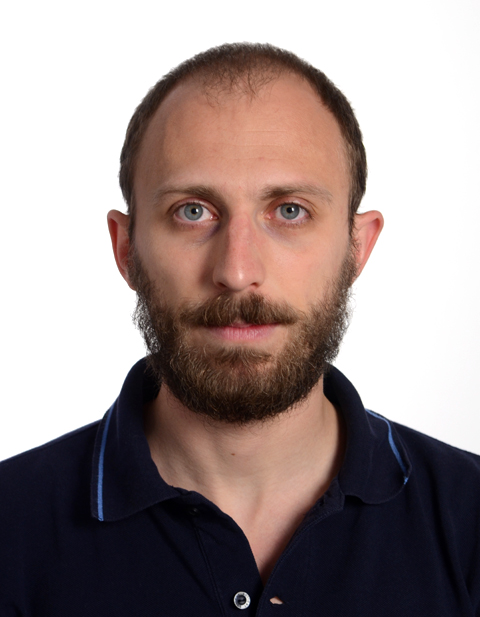}}]{Paolo Frasca}
(M'13, SM'18) received the Ph.D.\ degree in Mathematics for Engineering Sciences from Politecnico di Torino, Torino, Italy, in 2009. 
From 2013 to 2016, he was an Assistant Professor at the University of Twente in Enschede, the Netherlands. Since October 2016 he is a CNRS Researcher affiliated with GIPSA-lab, Grenoble, France. 
His research interests are in the theory of networks and control systems, with main applications to transportation and social networks. \end{IEEEbiography}

\end{document}